\documentclass[11pt]{amsart}
\usepackage{amsmath,amsxtra,amssymb,amsthm,amsfonts,showkeys}
\usepackage{color}

\numberwithin{equation}{section}
\newtheorem{lemma}{Lemma}[section]
\newtheorem{prop}[lemma]{Proposition}
\newtheorem{theorem}[lemma]{Theorem}

\newtheorem{cor}[lemma]{Corollary}

\newtheorem{rem}[lemma]{Remark}

\newcommand{\re}{\begin{rem}\rm}
  \newcommand{\mar}{\end{rem}}

\renewcommand{\for}{\begin{eqnarray*}}

\newcommand{\mel}{\end{eqnarray*}}

\newcommand{\nz}{{\mathbb N}}

\newcommand{\cz}{{\mathbb C}}

\newcommand{\ten}{\otimes}

\newcommand{\qd}{\end{proof}\vspace{0.5ex}}

\newcommand{\Om}{\Omega}

\newcommand{\A}{{\mathcal A}}

\newcommand{\M}{{\mathcal M}}

\newcommand{\U}{{\mathcal U}}

\newcommand{\pf}{\begin{proof}}

\newcommand{\xspace}{\hbox{\kern-2.5pt}}
\newcommand{\xyspace}{\hbox{\kern-1.1pt}}

\allowdisplaybreaks

\definecolor{LightGray}{rgb}{0.94,0.94,0.94}
\definecolor{VeryLightBlue}{rgb}{0.9,0.9,1}
\definecolor{LightBlue}{rgb}{0.8,0.8,1}
\definecolor{DarkBlue}{rgb}{0,0,0.6}
\definecolor{LightGreen}{rgb}{0.88,1,0.88}
\definecolor{MidGreen}{rgb}{0.6,1,0.6}
\definecolor{DarkGreen}{rgb}{0,0.6,0}
\definecolor{DarkGrreen}{rgb}{0,0.8,0}

\definecolor{VeryLightYellow}{rgb}{1,1,0.9}
\definecolor{LightYellow}{rgb}{1,1,0.6}
\definecolor{MidYellow}{rgb}{1,1,0.5}
\definecolor{DarkYellow}{rgb}{0.8,1,0.3}
\definecolor{VeryLightRed}{rgb}{1,0.9,0.9}
\definecolor{LightRed}{rgb}{1,0.8,0.8}
\definecolor{DarkRed}{rgb}{0.8,0.2,0}
\definecolor{DarkRedb}{rgb}{0.6,0.2,0}
\definecolor{DarkLila}{rgb}{0.8,0,1}
\definecolor{Beige}{rgb}{0.96,0.96,0.86}
\definecolor{Gold}{rgb}{1.,0.84,0.}
\definecolor{Goldb}{rgb}{0.7,0.3,0.5}
\definecolor{MyYellow}{rgb}{1.,0.84,0.8}

\begin{document}

\title[]{Generalized $q$-gaussian von Neumann algebras with coefficients, III. Unique prime factorization results.}
\author[Marius Junge]{Marius Junge}
\address{Department of Mathematics\\
University of Illinois, Urbana, IL 61801, USA} 
\email[Marius Junge]{junge@math.uiuc.edu}

\author[Bogdan Udrea]{Bogdan Udrea}
\address{Department of Mathematics\\
University of Iowa, Iowa City, IA 52242, USA} 
\email[Bogdan Udrea]{bogdanteodor-udrea@uiowa.edu}

\begin{abstract}We prove some unique prime factorization results for tensor products of type $II_1$ factors of the form $\Gamma_q(\cz, S \ten H)$ arising from symmetric independent copies with sub-exponential dimensions of the spaces $D_k(S)$ and dim$(H)$ finite and greater than a constant depending on $q$.
\end{abstract}

\maketitle
\section{Introduction.}
This article is a continuation of the program initiated in \cite{JungeUdreaGQC}. In \cite{JungeUdreaGQC} we introduced the generalized $q$-gaussian von Neumann algebras $\Gamma_q(B,S\ten H)$ with coefficients in $B$ and proved their strong solidity relative to $B$ under the assumptions of $dim_B(D_k(S))$ sub-exponential and $dim(H)<\infty$ (see \cite{JungeUdreaGQC}, Def. 3.18 and Cor. 7.4). In a subsequent paper (\cite{JungeUdreaACS}) we investigated the presence of non-trivial central sequences and we showed they do not exist when $B$ is a finite dimensional factor, the dimensions over $B$ of the modules $D_k(S)$ are sub-exponential and the dimension of $H$ is finite and greater than a constant depending on $q$. In the present work we prove some unique prime factorization results for tensor products of von Neumann algebras of the form $\Gamma_q(\cz,S\ten H)$ arising from a sequence of symmetric independent copies over $\cz$ and having sub-exponential dimensions (over $\cz$) of the spaces $D_k(S)$ introduced in \cite{JungeUdreaGQC}, Def. 3.18. The first results of this kind for type $II_1$ factors - arising from either (discrete) ICC non-amenable hyperbolic groups or (discrete) subgroups of connected simple Lie groups of rank one - have been obtained by Ozawa and Popa in \cite{OzawaPopaUPF}, through a combination of Ozawa's $C^*$-algebraic techniques previously used in \cite{OzawaSolid} and the powerful intertwining and unitary conjugacy techniques of Popa (see e.g. \cite{PoBe}, Appendix and \cite{PoI}, Thm. 2.1). Let's recall that for $(\M,\tau)$ a type $II_1$ factor and $t>0$, the amplification of $\M$ by $t$ is defined as $\M^t=p(M_n(\cz) \ten \M)p$, where $n>t$ and $p \in M_n(\cz) \ten \M$ is a projection with $\tau(p)=t/n$. Our main result is (see also Thm. 1 in \cite{OzawaPopaUPF}):
\begin{theorem} Let $M_k=\Gamma_{q_k}(\cz,S_k \ten H_k)$ coming from a sequence of symmetric independent copies $(\pi_j^k,\cz, A_k,D_k)$ for $-1<q_k<1$ and $1 \leq k \leq n$. Assume that for all $1 \leq k \leq n$, $H_k$ is finite dimensional and $dim_{\cz}((D_k)_i(S_k))<Cd^i$ for all $i$ and some constants $C,d>0$. Suppose that $M=\overline{\bigotimes}_{k=1}^n M_k=N_1 \bar{\ten} N_2$ for some type $II_1$ factors $N_1$ and $N_2$. Then there exists $t>0$ and a partition $I_1 \sqcup I_2=\{1,\ldots,n\}$ such that, modulo conjugacy by an unitary in $M$, we have $N_1^t=\overline{\bigotimes}_{k \in I_1} M_k$ and $N_2^{1/t}=\overline{\bigotimes}_{k \in I_2} M_k$.
\end{theorem}
To prove Thm. 1.1, instead of relying on $C^*$-algebraic techniques and the property (AO) as in \cite{OzawaSolid, OzawaPopaUPF}, we use our relative strong solidity result in \cite{JungeUdreaGQC}, Cor. 7.4. We should note that the von Neumann algebras $\Gamma_{q_k}(\cz,S_k\ten H_k)$ are automatically factors if $dim(H_k) \geq d(q_k)$ (Prop. 3.23 in \cite{JungeUdreaGQC}). By repeatedly applying Thm. 1.1 one obtains
\begin{cor}Let $M_k=\Gamma_{q_k}(\cz,S_k\ten H_k)$ with the dimensions (over $\cz$) of the spaces $(D_k)_i(S_k)$ sub-exponential and $\infty>dim(H_k)\geq d(q_k)$ for all $1 \leq k \leq n$. Assume that
\[M_1 \bar{\ten} \cdots \bar{\ten} M_n = N_1 \bar{\ten} \cdots \bar{\ten} N_m,\]
for $m \geq n$ and some type $II_1$ factors $N_1, \ldots, N_m$. Then $m=n$ and there exist $t_1,\ldots,t_n>0$ with $t_1t_2 \cdots t_n=1$ such that, after permutation of indices and unitary conjugacy, we have $N_k^{t_k}=M_k$.
\end{cor}
When the factors $N_j$ are assumed to be prime the assumption $m \geq n$ becomes unnecessary and hence we obtain
\begin{cor} Let $M_1, \ldots, M_n$ be generalized $q$-gaussians as above. Suppose that for some $m \in \nz$ and prime type $II_1$ factors $N_1, \ldots, N_m$ we have
\[M_1 \bar{\ten} \cdots \bar{\ten} M_n = N_1 \bar{\ten} \cdots \bar{\ten} N_m.\]
Then $m=n$ and there exist $t_1,\ldots,t_n>0$ with $t_1t_2 \cdots t_n=1$ such that, after permutation of indices and unitary conjugacy, we have $N_k^{t_k}=M_k$. In particular this holds if each $N_j=\Gamma_{q_j}(\cz,T_j \ten K_j)$ is a generalized $q_j$-gaussian with scalar coefficients, sub-exponential dimensions of $(D_j)_i(T_j)$ and $dim(K_j)<\infty$.
\end{cor}
In particular, $M_k$ and / or $N_j$ could be any one of the examples in 4.4.1, 4.4.2, 4.4.3 in \cite{JungeUdreaGQC}. Thus, if $M_i, 1\leq i \leq n$, and $N_j, 1\leq j \leq m$, are generalized $q$-gaussian von Neumann algebras as above and $m \neq n$, then $M_1 \bar{\ten} \cdots \bar{\ten} M_n \ncong N_1 \bar{\ten} \cdots \bar{\ten} N_m$.
\section{Proof of the main theorem.}
Throughout this section, we freely use notations and results from Section 3 of \cite{JungeUdreaGQC}. We start by stating some preliminary technical results. The first one is Prop. 2.7 in \cite{PoVaI}. If $(M,\tau)$ is a tracial von Neumann algebra and $P, Q \subset M$ are von Neumann subalgebras, we say that $P$ is amenable relative to $Q$ (inside $M$) if there exists a $P$-central state $\Om$ on $B(L^2(M)) \cap (Q^{op})'$ such that $\Om|_M=\tau$ (see e.g. Def. 2.2 in \cite{PoVaI}).
\begin{prop}Let $(M,\tau)$ be a tracial von Neumann algebras and let $Q_1, Q_2 \subset M$ be von Neumann subalgebras. Assume that $Q_1, Q_2$ form a commuting square, which means $E_{Q_1}\circ E_{Q_2}=E_{Q_2} \circ E_{Q_1}$, where $E_{Q_1},E_{Q_2}$ are the conditional expectations of $M$ onto $Q_1, Q_2$ respectively, and that $Q_1$ is regular in $M$. Let $P \subset M$ be a von Neumann subalgebra which is amenable relative to both $Q_1$ and $Q_2$. Then $P$ is amenable relative to $Q_1 \cap Q_2$.
\end{prop}
The next result is Prop. 12 in \cite{OzawaPopaUPF}.
\begin{prop} Let $M=M_1 \bar{\ten} M_2$ and $N \subset M$ be type $II_1$ factors. Assume that $N \prec_M M_1$ and $N'\cap M$ is a factor. Then there exists a decomposition $M=M_1^t \bar{\ten} M_2^{1/t}$ for some $t>0$ and a unitary $u \in \U(M)$ such that $uNu^* \subset M_1^t$.
\end{prop}
The next result will be needed in the proof of Thm. 1.1. It is an analogue of Prop. 15 in \cite{OzawaPopaUPF}. For convenience, if $M=M_1 \bar{\ten} \cdots \bar{\ten} M_n$ and $1 \leq k \leq n$, let's denote by
\[\widehat{M}_k = M_1 \bar{\ten} \cdots M_{k-1} \bar{\ten} 1 \bar{\ten} M_{k+1} \cdots \bar{\ten} M_n \subset M.\]
More generally, for every subset $I \subset \{1,\ldots,n\}$, we will denote by $\widehat{M}_I$ the von Neumann algebra
\[\widehat{M}_I=\overline{\bigotimes}_{i \notin I} M_i \subset M.\]
\begin{prop} Let $M_i=\Gamma_{q_i}(\cz, S_i\ten H_i)$ be generalized $q$-gaussian von Neumann algebras with scalar coefficients coming from symmetric independent copies and having sub-exponential dimensions over $\cz$ of the spaces $(D_i)_k(S_i)$, for all $1 \leq i \leq n$. Let $M=M_1 \bar{\ten} \cdots \bar{\ten} M_n$ and assume that $N \subset M$ is a type $II_1$ factor such that $N' \cap M$ is a non-amenable factor. Then there exists $t>0$, $1 \leq k \leq n$ and a unitary $u \in \U(M)$ such that $uNu^* \subset (\widehat{M}_k)^t$.
\end{prop}
\begin{proof}Let's first note that there exists a $1 \leq k \leq n$ such that $N'\cap M$ is not amenable relative to $\widehat{M}_k$. Indeed, if this were not the case, since the subalgebras $\widehat{M}_I, \widehat{M}_J$ form a commuting square for all subsets $I, J \subset \{1,\ldots,n\}$ and all of them are regular in $M$, by repeatedly applying Prop. 2.1, we would obtain that $N'\cap M$ is amenable relative to $\bigcap_{k=1}^n \widehat{M}_k=\cz$, i.e. 
$N'\cap M$ is amenable, a contradiction. Fix a $k$ such that $N'\cap M$ is not amenable relative to $\widehat{M}_k$. Suppose that $N \nprec_M \widehat{M}_k$. By Cor. F.14 in \cite{BrownOzawa} there exists an abelian von Neumann subalgebra $\A \subset N$ such that $\A \nprec_M \widehat{M}_k$. Let's make the following general remark. Suppose $\Gamma_q(B,S\ten H)$ is associated to a sequence of symmetric independent copies $(\pi_j,B,A,D)$ and let $\M$ be any tracial von Neumann algebra. Then
\[\M \bar{\ten} \Gamma_q(B,S\ten H)=\Gamma_q(B \bar{\ten} \M,S\ten H),\]
associated to a new sequence of symmetric independent copies $(\tilde{\pi}_j,\tilde B, \tilde A,\tilde D)$, defined by $\tilde B=B\bar{\ten} \M$, $\tilde A=A \bar{\ten} \M$, $\tilde D=D \bar{\ten} \M$ and 
$\tilde{\pi}_j:\tilde A \to \tilde D$ are given by $\tilde{\pi}_j(a \ten x)=\pi_j(a) \ten x$, for $a \in A, x \in \M$. 
Now note that
\[\A \subset M=\widehat{M}_k \bar{\ten} M_k=\widehat{M}_k \bar{\ten} \Gamma_{q_k}(\cz,S_k\ten H_k) = \Gamma_{q_k}(\widehat{M}_k,S_k\ten H_k).\]
It's trivial to check that $dim_{\widehat{M}_k}((D_k)_i(S_k))=dim_{\cz}(D_k)_i(S_k)$ are sub-exponential. Since $\A$ is amenable relative to $\widehat{M}_k$, by applying Cor. 7.4 in \cite{JungeUdreaGQC}, we must have that either $\A \prec_M \widehat{M}_k$ or $N_M(\A)''$ is amenable relative to $\widehat{M}_k$. The first half of the alternative is precluded by the choice of $\A$, and the second would imply that $N'\cap M \subset N_M(\A)''$ is also amenable relative to $\widehat{M}_k$, which is a contradiction. Thus $N \prec_M \widehat{M}_k$, and by Prop. 2.2 we see that there exists an unitary $u \in \U(M)$ and a $t>0$ such that $uNu^* \subset (\widehat{M}_k)^t$.
\end{proof}
Now we can prove Thm. 1.1. The proof proceeds verbatim as in \cite{OzawaPopaUPF}. We nevertheless give details for completeness.
\begin{proof}[proof of Theorem 1.1.] 
We use induction over $n$. The case $n=0$ is trivial. Let $M=\overline{\bigotimes}_{k=1}^n M_k=N_1 \bar{\ten} N_2$. Since $M$ is non-amenable, we can assume that $N_2$ is non-amenable. By Prop. 2.3 there exist $t>0$, $1 \leq k \leq n$ and $u \in \U(M)$ such that $uN_1u^* \subset (\widehat{M}_k)^t$. Set $\M_1=\widehat{M}_k$, $\M_2=M_k$ and $N_{2,1}=N_1'\cap u\M_1^tu^*$. Then we see that
\[N_2=N_1'\cap M=u^*(N_{2,1} \bar{\ten} \M_2^{1/t})u \subset u^*(\M_1^t \bar{\ten} \M_2^{1/t})u=M,\]
and $u\M_1^tu^*$ is generated by $N_1$ and $N_{2,1}$. Using the induction hypothesis, we can find an $s>0$, a partition $I_1, I_{2,1}$ of $\{1,\ldots,n\} \setminus \{k\}$ such that $N_1=(\overline{\bigotimes}_{j \in I_1} M_j)^s$ and $N_{2,1}=(\overline{\bigotimes}_{j \in I_{2,1}} M_j)^{t/s}$ after conjugating with a unitary element in $(\widehat{M}_k)^t$. If we now set $I_2=I_{2,1} \cup \{k\}$, the proof is complete.
\end{proof}
\begin{rem}The statement of the results and the proofs remain verbatim the same if one assumes that $B_k$ is a finite dimensional factor for every $1 \leq k \leq n$.
\end{rem}

\bibliographystyle{amsplain}

\bibliographystyle{amsplain}
\bibliography{thebibliography}
\end{document}